\newlength\bibsep
\documentclass[final,11pt,a4paper,3p,times,nonatbib]{elsarticle}

\usepackage{amssymb}
\usepackage{mathrsfs}
\usepackage{amsthm}
\usepackage{booktabs}
\usepackage{multirow}

\usepackage{algorithm}
\usepackage[noend]{algorithmic}
\usepackage{color,xcolor}
\usepackage{amsmath}
\usepackage{apacite}
\usepackage{url}
\usepackage{setspace}
\newtheorem{theorem}{Theorem}
\usepackage{subfig}
\usepackage{array}
\AtBeginEnvironment{algorithm}{\onehalfspacing}
\AtBeginEnvironment{algorithmic}{\onehalfspacing}     
\allowdisplaybreaks
\usepackage{threeparttable}
\newtheorem{proposition}{Proposition}
\renewcommand\theparagraph{\thesubsubsection.\arabic{paragraph}}
\makeatother
\usepackage{titlesec}
\titleformat{\paragraph}[block]{\normalfont\normalsize\itshape}{\theparagraph}{1em}{}
\titlespacing{\paragraph}{0pt}{1ex plus .2ex}{0.6ex}
\newtheorem{lemma}{Lemma}

\usepackage{lineno}

\journal{}

\begin{document}
\begin{frontmatter}
		
		

		\title{A Combinatorial Benders Decomposition Framework for Two-Dimensional Irregular Bin Packing Problems with Convex Polygons}
		
		\author[]{Jianming Wang$^{a}$,  Zhouwang Yang$^{a,}$\corref{cor}{}}
		\cortext[cor]{Corresponding author. 
		{\it E-mail address}: yangzw@ustc.edu.cn (Zhouwang Yang)}
		\affiliation[USTC]{organization={University of Science and Technology of China},
			city={Hefei},
			postcode={230026}, 
			country={P. R. China}}
		

		\begin{abstract}
		Two-dimensional irregular bin packing combines combinatorial bin-assignment decisions with difficult geometric feasibility constraints, making exact optimization challenging. This paper develops an exact combinatorial Benders decomposition framework for the two-dimensional irregular bin packing problem with convex polygons, coupling a pattern-based master problem with an exact single-bin geometric feasibility oracle. The dynamically strengthened Benders master is solved by a tailored exact branch-and-price procedure that incorporates objective-layered search and adaptive exact pricing. Geometric information obtained from the oracle is further fed back to the master and pricing processes through dynamically generated Benders feasibility cuts, which progressively restrict the subsequent pricing problems.  Computational experiments are conducted on 540 benchmark instances from 18 classes. The proposed method obtains the optimal solution for 318 instances across 12 classes within a 3600-second time limit.
        On these classes, the proposed method solves more instances than a direct mixed-integer programming formulation and a baseline combinatorial Benders decomposition approach. 
		\end{abstract}
		
		
%
		\begin{keyword}
			Two-dimensional irregular bin packing; Combinatorial Benders decomposition;
Branch-and-price; Exact optimization; Irregular packing
			
			
			
		\end{keyword}
		
	\end{frontmatter}
	
	

\section{Introduction}

Two-dimensional irregular bin packing problems arise in a wide range of industrial applications, including sheet metal cutting, furniture production, and garment manufacturing. In these problems, a set of irregular polygons must be packed into a minimum number of identical rectangular bins while satisfying containment and non-overlap constraints \shortcite{wascher2007improved}. Compared with rectangular bin packing, the irregular case combines the combinatorial difficulty of bin assignment with the geometric complexity of feasible placement, making exact optimization particularly challenging~\shortcite{leao2020irregular,guo2022two}. In this paper, we consider the two-dimensional irregular bin packing problem in which all items are represented as convex polygons and rotations are not allowed.

The literature on irregular packing is dominated by heuristic, metaheuristic, and matheuristic approaches. For benchmark instances of irregular bin packing and related variants, previous studies have developed effective constructive heuristics, local search methods, and matheuristics~\shortcite{lopez2013effective,martinez2017matheuristics,zhang2022iteratively,guerriero2023hierarchical}. These approaches have substantially improved the ability to obtain high-quality feasible solutions for irregular packing instances. However, they do not provide optimality guarantees. The exact solvability of standard benchmark instances remains  limited, particularly in terms of obtaining provable optimality certificates.

Exact optimization has advanced much further for rectangular packing problems than for irregular packing problems. For two-dimensional rectangular bin packing problems, branch-and-price, branch-and-cut, and decomposition-based approaches have demonstrated strong capabilities by exploiting structured assignment models and effective feasibility evaluations~\shortcite{iori2021exact,cote2021combinatorial}. In particular, combinatorial Benders decomposition separates assignment decisions from geometric feasibility and iteratively adds cuts derived from infeasible packing configurations ~\shortcite{cote2021combinatorial}. 
For irregular packing problems, exact studies have mainly focused on related settings such as irregular strip packing ~\shortcite{alvarez2013branch,cherri2016robust,Cherri2019optimality,lastra2024mixed}. However, existing exact approaches for irregular bin packing problems either rely on predefined discrete placement candidates \shortcite{Cherri2019optimality}, or address related but different packing settings. Therefore, a systematic exact framework for the standard two-dimensional irregular bin packing problem, where item-to-bin assignment and irregular geometric feasibility must be handled simultaneously, remains limited.

To address this gap, we propose an exact combinatorial Benders decomposition framework for two-dimensional irregular bin packing with convex polygons. The framework separates bin assignment from geometric feasibility by coupling a pattern-based master problem with an exact single-bin oracle. 
The master problem is solved through a tailored branch-and-price procedure. Single-bin feasibility is checked exactly using a vertical-slice MIP formulation based on no-fit polygons. 
Geometric infeasibility identified by the oracle is progressively fed back into the search through valid Benders cuts.
This design enables the combinatorial and geometric components to be handled separately while maintaining exactness.

The main contributions of this paper are as follows.
First, we develop an exact combinatorial Benders decomposition framework for two-dimensional irregular bin packing, coupling a pattern-based master problem with an exact single-bin geometric feasibility oracle.
Second, we design a tailored exact branch-and-price procedure for the restricted Benders master problem, incorporating objective-layered search and adaptive exact pricing.
Third, we conduct a systematic study on 540 benchmark instances from 18 classes, proving optimality for 318 instances across 12 classes and showing stronger performance than a direct MIP formulation and a baseline combinatorial Benders approach on most of these classes.

The remainder of the paper is organized as follows. Section~\ref{sec:related} reviews the most relevant literature. Section~\ref{sec:benders-formulation} introduces the problem formulation.  Section~\ref{sec:method} presents the proposed exact decomposition method. Section~\ref{sec:experiments} reports the computational experiments and discusses the results. Section~\ref{sec:conclusion} concludes the paper and outlines future research directions.

\section{Related Work}
\label{sec:related}

\subsection{Two-Dimensional Irregular Packing and Bin Packing Problems}
Irregular packing problems form a broad family of optimization problems in which non-rectangular items must be placed inside one or multiple bins while satisfying containment and non-overlap constraints. 
A comprehensive review of mathematical models for irregular packing is given by~\shortciteA{leao2020irregular}, who emphasize the diversity of problem settings, geometric representations, and optimization objectives arising in  strip packing and bin packing problems. 
As highlighted in that review, the irregular bin packing variant is structurally distinct from strip packing and related single-bin problems because it simultaneously requires assignment decisions across multiple bins and geometric feasibility decisions within each bin.

For two-dimensional irregular bin packing, most of the existing literature has focused on heuristic, metaheuristic, or matheuristic approaches. Early work by \shortciteA{lopez2013effective} developed an effective heuristic for the problem and helped establish benchmark-oriented experimental practice. Later studies further improved solution quality by incorporating richer placement strategies, local search, hyper-heuristics, and matheuristics. In particular, \shortciteA{martinez2017matheuristics} proposed matheuristics for the irregular bin packing problem with free rotations, \shortciteA{zhang2022iteratively} developed an iteratively doubling local search for limited-rotation instances, \shortciteA{guerriero2023hierarchical} introduced a hierarchical hyper-heuristic, \shortciteA{cai2023heuristics} proposed a recent heuristic framework for two-dimensional irregular bin packing with limited rotations, and \shortciteA{luo2026permutation} developed a permutation-coded evolutionary algorithm that emphasizes global search over packing permutations and reports strong computational performance. These studies demonstrate that high-quality feasible solutions can be obtained in practice, but they do not provide exact optimality guarantees.

The present work is positioned in this line of research at the level of problem class and benchmark setting, but differs in methodology. Rather than designing another heuristic or matheuristic for benchmark instances of two-dimensional irregular bin packing, we investigate exact optimization for the standard benchmark 2DIBPP.

\subsection{Exact Methods for Packing Problems}

Exact optimization for two-dimensional irregular bin packing requires addressing two closely coupled challenges: combinatorial assignment across multiple bins and continuous geometric feasibility within each bin. 
Existing exact methods have made substantial progress on these two aspects, but largely in different problem settings. 
In particular, two lines of research are most relevant to the present study: exact methods for two-dimensional rectangular bin packing and exact formulations for irregular packing. 

The first stream concerns exact methods for  two-dimensional rectangular bin packing problems. 
\shortciteA{martello1998exact} developed one of the first exact algorithms for the two-dimensional finite bin packing problem, establishing important foundations for later branch-and-bound and decomposition approaches.
Later surveys by \shortciteA{lodi2002recent}  highlighted the central role of lower bounds, enumeration, and decomposition in exact algorithms for two-dimensional bin packing. 
More recent surveys, such as \shortciteA{iori2021exact}, show how exact methods for orthogonal packing evolved toward branch-and-bound, branch-and-price, and solver-based hybrid frameworks. 
Within this line, \shortciteA{pisinger2007using} developed a decomposition-based exact algorithm combining branch-and-price ideas with constraint programming for the two-dimensional rectangular bin packing problem. 
More recently, \shortciteA{cote2021combinatorial} proposed a combinatorial Benders decomposition approach in which a descriptive master problem is iteratively strengthened using infeasible packing patterns identified through geometric feasibility checks. 
In a different direction, \shortciteA{wang2025highly} developed a numerically exact branch-and-price algorithm for the two-dimensional rectangular bin packing problem and its guillotine variant. Their method adopts a pattern-based master formulation and generates geometrically feasible two-dimensional packing patterns through a specialized exact pricing procedure. 
These studies demonstrate that decomposition and branch-and-price can effectively handle the combinatorial structure of multi-bin packing. 
However, their geometric components rely heavily on the special structure of rectangular items, and therefore do not directly address the continuous-space feasibility of irregular polygons considered in this paper.

The second stream focuses on exact methods for irregular two-dimensional packing, especially strip packing. 
In this area, geometric feasibility is considerably harder because non-overlap must be modeled through no-fit polygons, inner-fit polygons, or related continuous geometric constructions. 
\shortciteA{alvarez2013branch} proposed an exact branch-and-bound framework for irregular cutting and packing. 
\shortciteA{cherri2016robust} developed robust mixed-integer linear programming models for the irregular strip packing problem with geometric preprocessing and strengthened formulations.
More recently, \shortciteA{lastra2024mixed} proposed exact mixed-integer formulations based on vertical slices and feasibility cuts, further advancing solver-based exact optimization for irregular strip packing. 
In addition to continuous geometric formulations, another important direction for exact irregular packing relies on discretized placement representations.
\shortciteA{Toledo2013dotted} introduced the dotted-board model, where feasible placements are restricted to a finite set of candidate positions on a discretized board. 
Following this direction, \shortciteA{Cherri2019optimality} developed constraint programming models and a global non-overlap constraint for several irregular cutting and packing variants, including the irregular bin packing problem (IBPP). 
These studies establish exact tools for irregular geometric feasibility, but they have been developed mainly for strip or single-bin settings, or for multi-bin models based on discretized placement spaces.

Overall, existing exact approaches provide strong methodological building blocks for either multi-bin combinatorial optimization or irregular geometric feasibility, but exact solution frameworks for the standard two-dimensional irregular bin packing problem remain limited.
A recent conference abstract by \shortciteA{irnich2026column} considers irregular-shape bin packing with rectangular bins of various sizes using a column-generation framework. 
In their approach, columns represent geometrically feasible nestings, while the pricing problem is first relaxed to a bounded knapsack problem.
Candidate patterns are subsequently validated by exact or heuristic nesting procedures, and infeasible candidates are excluded through no-good cuts. 
In contrast, our framework uses branch-and-price to solve the restricted Benders master problem, while geometric feasibility is handled independently by an exact single-bin oracle. 
Since only a conference abstract is currently available, a systematic description of the full methodology has not yet been published.

\subsection{Methodological Foundations}

The proposed framework builds on two main methodological foundations: Benders decomposition for separating high-level assignment decisions from lower-level feasibility decisions, and branch-and-price techniques for solving the resulting pattern-based master problem.

Classical Benders decomposition~\shortcite{benders1962partitioning} separates a problem into a master problem and one or more subproblems, but in its traditional form it is naturally suited to continuous subproblems. 
Logic-based Benders decomposition (LBBD), introduced by \shortciteA{hooker2003logic}, generalizes this idea to settings in which subproblems are combinatorial or constraint-programming models. \shortciteA{hooker2007planning} and \shortciteA{cire2016logic} showed that LBBD can be especially effective for planning and scheduling problems when subproblems become much easier after fixing high-level master decisions.
Within packing-related exact optimization, Benders decomposition has been successfully applied to rectangular packing. 
In particular, \shortciteA{cote2021combinatorial} developed a combinatorial Benders decomposition framework for two-dimensional rectangular bin packing, separating bin-assignment decisions from geometric packing feasibility.
This study motivates the use of feasibility-driven Benders decomposition in our framework.
More recently, \shortciteA{nascimento2026improving} applied logic-based Benders decomposition to a scheduling problem with two-dimensional packing. Although related in methodology, their work focuses on an integrated scheduling--packing problem rather than a pure two-dimensional irregular bin packing problem.

The master side of our framework is closely connected to exact methods for one-dimensional bin packing. A classical milestone is the branch-and-price work of \shortciteA{valerio1999exact}, which established the central role of set-covering models and column generation in exact bin packing. 
This line was later systematized by \shortciteA{delorme2016bin}, who reviewed the main mathematical models and exact algorithms for one-dimensional bin packing problems.
They clarify the relationships among pattern-based formulations, pseudo-polynomial models, and decomposition-based exact methods.
More recent work has significantly strengthened this line in two directions. On the one hand, graph-based and pseudo-polynomial formulations have been improved. 
In particular, \shortciteA{brandao2016bin} proposed a general arc-flow formulation with graph compression, while \shortciteA{delorme2020enhanced} developed enhanced pseudo-polynomial formulations for bin packing. 
On the other hand, modern branch-and-price algorithms have continued to advance exact solution methods. \shortciteA{wei2020new} proposed a branch-and-price-and-cut algorithm for one-dimensional bin packing, combining strong valid inequalities with a specialized label-setting pricing procedure. More recently, \shortciteA{baldacci2024numerically} developed a numerically exact branch-price-and-cut algorithm for bin packing, with particular attention to numerical reliability.

From a modeling perspective, each restricted Benders master problem in our framework can be viewed as a variant of the one-dimensional bin packing problem augmented with constraints that exclude geometrically infeasible item combinations. 
This connection allows branch-and-price and exact pricing techniques developed for one-dimensional bin packing to be adapted to the restricted Benders master problem. 
At the same time, unlike standard bin packing, the admissible pattern space evolves dynamically as geometric infeasibility is identified by the single-bin oracle and fed back through Benders cuts. 
The proposed framework therefore combines Benders decomposition with branch-and-price in a tightly coupled manner: the former provides the overall decomposition structure, while the latter serves as the exact solution mechanism for the dynamically strengthened master problem.

\section{Benders Decomposition Formulation}
\label{sec:benders-formulation}

Let $I=\{1,\ldots,n\}$ be the set of items and let $B=\{1,\ldots,n\}$ be
the set of potential bins. Each item $i\in I$ has area $A_i$, and each bin
has area $A$. A pattern represents a subset of items assigned to one bin.
Let $\mathcal{P}$ denote the set of all candidate item subsets satisfying the basic bin-capacity condition, and let $a_{ip}$
be equal to one if item $i$ belongs to pattern $p$, and zero otherwise.
For each pattern $p\in\mathcal{P}$, let
$S(p)=\{i\in I:a_{ip}=1\}$.

The geometric feasibility of a pattern is represented by the oracle
$\Phi(S)$, where
\[
\Phi(S)=
\begin{cases}
1, & \text{if the items in } S \text{ can be packed into one bin},\\
0, & \text{otherwise}.
\end{cases}
\]

Let $\mathcal{F}$ denote the family of all the geometrically infeasible item
subsets:
\[
    \mathcal{F}=\{S\subseteq I:\Phi(S)=0\}.
\]

We introduce binary variables $y_b$, which take value 1 if bin $b\in B$
is used, and binary variables $x_{ib}$, which take value 1 if item
$i\in I$ is assigned to bin $b\in B$. The two-dimensional irregular bin
packing problem can then be written as the following combinatorial
Benders reformulation:
\begin{align}
    \min \quad
        & \sum_{b\in B} y_b
        \label{eq:benders-obj} \\
    \text{s.t.}\quad
        & \sum_{b\in B} x_{ib} = 1,
        && \forall i\in I,
        \label{eq:benders-assign} \\
        & \sum_{i\in I} A_i x_{ib} \le A y_b,
        && \forall b\in B,
        \label{eq:benders-area} \\
        & \sum_{i\in S} x_{ib} \le |S|-1,
        && \forall b\in B,\ \forall S\in\mathcal{F},
        \label{eq:benders-feasibility-cuts} \\
        & x_{ib}\in\{0,1\},
        && \forall i\in I,\ \forall b\in B,
        \label{eq:benders-x-binary} \\
        & y_b\in\{0,1\},
        && \forall b\in B.
        \label{eq:benders-y-binary}
\end{align}

Constraint~\eqref{eq:benders-assign} ensures that each item is assigned
to exactly one bin. Constraint~\eqref{eq:benders-area} links item
assignments with bin activation and imposes the necessary area condition
for each open bin. Constraint~\eqref{eq:benders-feasibility-cuts} is the
combinatorial Benders cut family.
Whenever an item subset $S$ is
geometrically infeasible, the model forbids assigning all items in $S$ to
the same bin.

\section{Methodology}
\label{sec:method}

\subsection{Exact Combinatorial Benders Framework}
\label{subsec:benders-framework}

The family $\mathcal{F}$ is extremely large, and therefore it is neither
possible to enumerate all its elements nor practical to solve the complete
formulation with all feasibility constraints included. The formulation in
Section~\ref{sec:benders-formulation} should therefore be understood as an
implicit model. 
It characterizes the complete feasible region, but only a
small fraction of its feasibility constraints can be handled explicitly
during the solution process. 
To address this difficulty, we develop an exact combinatorial Benders decomposition that solves a restricted Benders master problem (RBMP), which contains only a subset of the
feasibility constraints, and  iteratively adds violated constraints as
needed. 
These dynamically generated constraints are referred to as
Benders cuts. Let $\mathcal{F}'\subseteq\mathcal{F}$ denote the subset of
infeasible item sets whose cuts have already been added to the RBMP. At
each iteration, the RBMP is solved over the current cut pool
$\mathcal{F}'$. The resulting solution defines a tentative assignment of
items to bins. 
For each open bin, we extract the corresponding item set
and call the single-bin geometric feasibility oracle to determine whether
this set can be packed into one bin.  
If the oracle certifies all selected item sets as feasible, the resulting solution is feasible for the complete formulation. Otherwise, a cut of the form given in constraint \ref{eq:benders-feasibility-cuts} is added to the RBMP, which is then re-solved.

Algorithm~\ref{alg:benders-framework} summarizes the resulting exact
combinatorial Benders framework, as illustrated in Figure~\ref{fig:frame}.
The algorithm first performs preprocessing, initializes the Benders cut pool, and computes a valid initial lower bound $LB_0$ , which defines the first objective layer $k$ (Lines 1 to 4).
At each iteration, the current RBMP is searched for an integer solution with objective value $k$ using the branch-and-price procedure described in Section~\ref{subsec:master-problem-solution} (Line 6). 
If no such RBMP solution exists, the current objective layer is certified infeasible and $k$ is increased by one (Lines 7 to 9).
Otherwise, the item set assigned to each open bin is extracted and checked by the exact single-bin feasibility oracle described in Section~\ref{subsec:single-bin-subproblem} (Line 10). 
If any selected item set is geometrically infeasible, the detected infeasible subsets are used to generate Benders feasibility cuts, which are added to the cut pool, and the associated pattern and pricing information is updated before the RBMP is searched again at the same objective layer (Lines 12 to 15 and Lines 19 to 21). 
If all selected item sets are certified feasible, the current solution is feasible for the complete formulation (Lines 16 and 17). 
Since $k$ is a valid global lower bound and all smaller objective layers have already been excluded, the solution is globally optimal and the algorithm terminates.
\begin{algorithm}[htb]
\caption{Exact combinatorial Benders framework.}
\label{alg:benders-framework}
\begin{algorithmic}[1]
\STATE Perform preprocessing and initialize the cut pool $\mathcal{F}'$.
\STATE Build the initial restricted Benders master problem (RBMP).
\STATE Compute a valid initial lower bound $LB_0$.
\STATE Set the current objective layer $k \gets LB_0$.

\WHILE{\textsc{True}}
    \STATE Search the current RBMP for an integer solution with objective value $k$
    using exact branch-and-price.
    
    \IF{no integer RBMP solution exists at objective layer $k$}
        \STATE Set $k \gets k+1$.
        \STATE \textbf{continue}.
    \ENDIF

    \STATE Extract the item set $S_b$ assigned to each open bin $b$.
    \STATE Set $\mathcal{V} \gets \emptyset$.

    \FOR{each selected open bin $b$}
        \STATE Call the exact single-bin feasibility oracle on $S_b$.
        \IF{$S_b$ is geometrically infeasible}
            \STATE Add the detected infeasible subset $\bar{S}_b$ to $\mathcal{V}$.
        \ENDIF
    \ENDFOR

    \IF{$\mathcal{V}=\emptyset$}
        \STATE Return the current RBMP solution as an optimal solution.
    \ELSE
        \STATE Select feasibility cuts $\mathcal{S}_{cut}$ based on $\mathcal{V}$.
        \STATE Update $\mathcal{F}' \gets \mathcal{F}' \cup S_{cut}$.
        \STATE Update the pattern pool and pricing information accordingly.
    \ENDIF
\ENDWHILE
\end{algorithmic}
\end{algorithm}

\subsubsection{Preprocessing and Initial Bounds}
\label{subsubsec:preprocessing-bounds}

Before solving the restricted Benders master problem, we initialize the
cut pool and the first lower bound used by the
objective-layered search. These preprocessing steps do not change the
feasible region of the complete formulation. They only add feasibility
cuts that are certified by the exact single-bin oracle and construct a
safe initial restricted master problem.

First, we identify pairwise geometric incompatibilities. For every pair
of items $\{i,j\}$, the exact single-bin feasibility oracle described in
Section~\ref{subsec:single-bin-subproblem} is called to determine whether
the two items can be packed together in one bin. If
$\Phi(\{i,j\})=0$, then the pair cannot appear in the same selected
pattern. The corresponding infeasible pair is added to the initial cut
pool:
\begin{equation}
    \mathcal{F}^{\mathrm{pair}}
    =
    \bigl\{\{i,j\}: i,j\in I,\ i<j,\ \Phi(\{i,j\})=0\bigr\}.
    \label{eq:pairwise-incompatibility}
\end{equation}
Since each pairwise cut is certified by the exact geometric oracle, this
step removes only patterns that are infeasible in the complete
formulation.

The initial restricted Benders master problem is then solved using the preprocessing cut pool
$\mathcal{F}^{\mathrm{pair}}$. During this solve, patterns may
be generated by the pricing procedure, while all generated patterns are
subject to the currently available Benders cuts. Let $z_{\mathrm{RBMP}}^0$ denote the optimal
value of this initial restricted master problem. Since the RBMP contains
only a subset of the complete Benders feasibility cuts, its optimal value
provides a valid lower bound on the original problem:
\begin{equation}
    LB_0 = \left\lceil z_{\mathrm{RBMP}}^0-\varepsilon \right\rceil,
    \label{eq:initial-lower-bound}
\end{equation}
where $\varepsilon$ is a small numerical tolerance.
This initial lower bound defines the starting objective level for the
objective-layered search described in
Section~\ref{subsubsec:objective-layered}.
\subsubsection{Combinatorial Benders Cuts and Geometric Feedback}
\label{subsubsec:benders-cuts-feedback}

Since enumerating all such infeasible subsets in advance is impractical, the algorithm maintains a dynamic cut pool and adds feasibility cuts only when the corresponding infeasible subsets are detected.
\begin{figure}[H]
    \centering
    \includegraphics[width=0.8\textwidth]{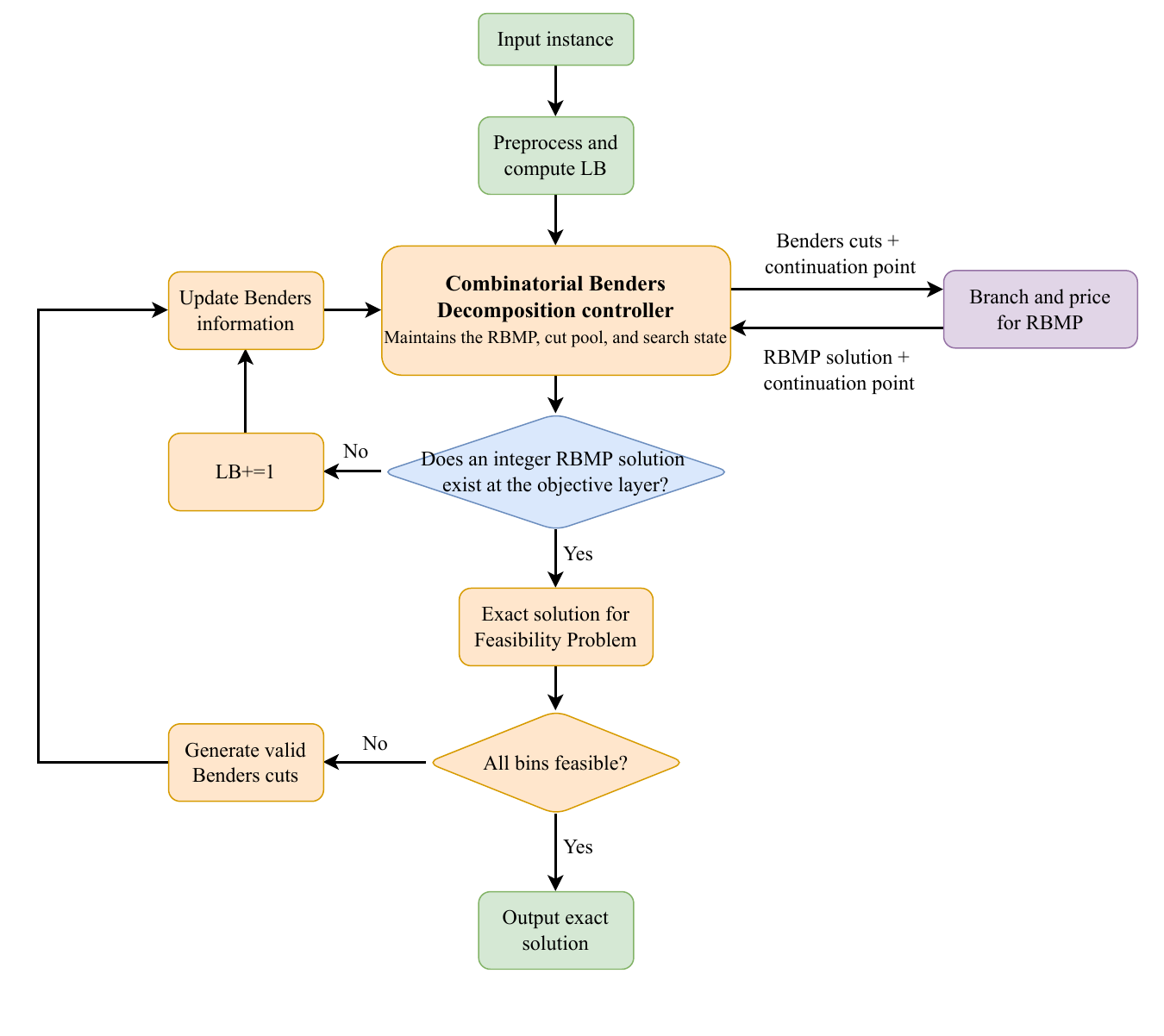}
    \caption{Overall structure of the proposed exact combinatorial Benders framework.}
    \label{fig:frame}
\end{figure}

Cuts are generated from integer solutions of the current RBMP. Suppose
that the RBMP returns a candidate solution with selected bins
$\mathcal{B}^{+}$. For each selected bin $b\in\mathcal{B}^{+}$, let
\begin{equation}
    S_b=\{i\in I:\text{item }i\text{ is assigned to bin }b\}
    \label{eq:selected-bin-item-set}
\end{equation}
be the item set assigned to that bin. The exact single-bin feasibility
oracle is then called on selected candidate bins. If $\Phi(S_b)=1$, the
bin is certified as geometrically feasible. If $\Phi(S_b)=0$, then the
algorithm derives an infeasible item set $\bar{S}_b\subseteq S_b$ and
adds it to the cut pool:
\begin{equation}
    \mathcal{F}' \leftarrow \mathcal{F}'\cup\{\bar{S}_b\}.
    \label{eq:dynamic-cut-pool-update}
\end{equation}
The corresponding Benders cut forbids every future pattern that contains
$\bar{S}_b$.
Figure \ref{fig:cut} illustrates how geometric infeasibility identified by the single-bin oracle is transformed into a combinatorial Benders cut that excludes all patterns containing the certified infeasible subset.
\begin{figure}[H]
    \centering
    \includegraphics[width=\textwidth]{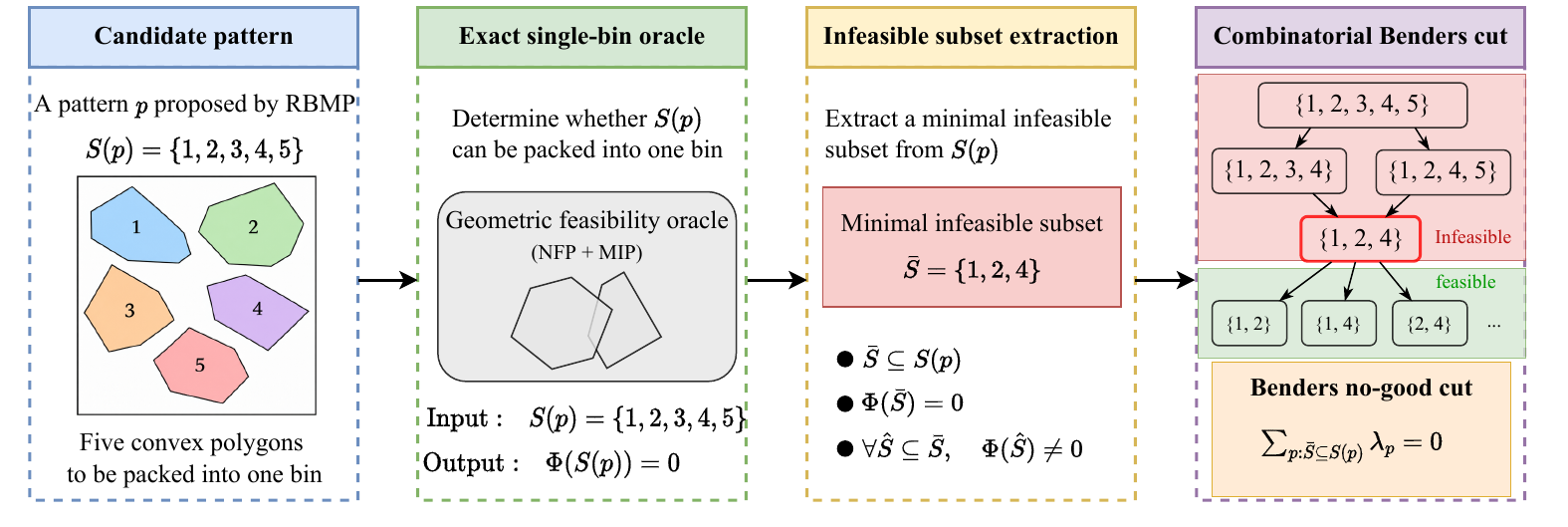}
    \caption{Geometric infeasibility feedback in the combinatorial Benders framework.}
    \label{fig:cut}
\end{figure}

In our approach, the cut-generation process follows a prioritized strategy. Candidate bins are first ordered by the number of assigned items, so that bins containing fewer items are checked earlier.
The number of inserted
infeasible cuts is limited by
\begin{equation}
    N_{cut}(\mathcal{B}^{+})
    =
    \max\left\{
        \left\lfloor \frac{|\mathcal{B}^{+}|}{2} \right\rfloor,
        1
    \right\}.
    \label{eq:limited-cut-number}
\end{equation}
The limit is imposed on the number of detected infeasible item sets,
rather than directly on the number of oracle calls. Therefore, the
algorithm may check more than $N_{cut}(\mathcal{B}^{+})$ candidate bins if some
checked bins are feasible, but it inserts at most $N_{cut}(\mathcal{B}^{+})$
new infeasibility cuts in one call. This strategy reduces the cost of
repeated exact single-bin feasibility checks while still providing
geometric feedback to the master problem. 
For infeasible bins with larger cardinality,
the whole item set $S_b$ is directly used as a no-good Benders cut. For
small infeasible bins, the algorithm applies an iterative deletion-based
refinement procedure. Starting from $S_b$, it attempts to remove one item
at a time and rechecks the remaining set with the single-bin feasibility
oracle. If the remaining set is certified infeasible, the item is
permanently removed and the procedure is repeated.
The resulting set $\bar{S}_b$ satisfies
\begin{equation}
    \bar{S}_b \subseteq S_b,
    \qquad
    \Phi(\bar{S}_b)=0,
\end{equation}
and is therefore sufficient to generate a valid Benders feasibility cut.
When all refinement oracle calls are solved to certification, the
procedure returns an inclusion-minimal infeasible subset. In general, the
refinement is used as a cut-strengthening procedure; the validity of the
generated cut follows from the certified infeasibility of the final
subset.
Smaller cuts are stronger because they forbid more infeasible patterns,
whereas whole-set cuts are cheaper to generate. Prioritizing candidate
item sets with smaller cardinality further increases the chance of
obtaining compact cuts, because fewer oracle rechecks are needed when
attempting to remove items from a smaller set. In our implementation, an
infeasible item set is considered to have small cardinality if it contains
fewer than eight items.

After new cuts are generated, they are inserted into the current cut pool, and the current pattern pool is filtered. Given the updated cut pool, a pattern \(p\) is admissible only if it does not contain any known infeasible subset:

$$
S \nsubseteq S(p),
\quad \forall S\in\mathcal{F}'.
$$
Equivalently, for each known infeasible subset \(S\in\mathcal{F}'\), all patterns containing \(S\) are forbidden by the Benders feasibility cut.
Accordingly, patterns containing newly detected infeasible subsets are removed, and the updated cut pool is enforced in the pricing procedure. 
As a result, these infeasible item combinations are excluded from subsequent master solutions and newly generated columns.

This mechanism forms the geometric feedback loop of the proposed framework. The RBMP proposes combinatorial item groupings, while the exact single-bin oracle certifies their geometric feasibility. Whenever an infeasible grouping is identified, the corresponding Benders cut is fed back to the master and subsequent pricing problems.

\subsection{Single-Bin Geometric Feasibility Subproblem}
\label{subsec:single-bin-subproblem}

The subproblem in our Benders framework is an exact single-bin geometric
feasibility problem. Given an item subset $S\subseteq I$, the oracle
determines whether all items in $S$ can be placed inside one bin without
overlap, while keeping their orientations fixed. This subproblem is used only as a feasibility oracle.
It does
not optimize the number of bins.

For each item $i\in S$, let $(u_i,v_i)$ denote the translation of its
reference point. Before building the feasibility model, the input
polygons are normalized and item-wise translation bounds are computed
with respect to the bin dimensions. Let $W$ and $H$ be the bin width and
height. The bounds
\begin{equation}
    \underline{u}_i \le u_i \le \overline{u}_i,
    \qquad
    \underline{v}_i \le v_i \le \overline{v}_i,
    \quad \forall i\in S,
    \label{eq:single-bin-placement-bounds}
\end{equation}
ensure that each individual item remains inside the bin. These bounds
define the domain of the placement variables used by the single-bin
oracle.
The non-overlap constraints are modeled using no-fit-polygon information.
For each pair of items $i,j\in S$, $i<j$, the relative placement vector
is
\begin{equation}
    \delta_{ij}=(u_j-u_i,\ v_j-v_i).
    \label{eq:relative-placement-vector}
\end{equation}
The items $i$ and $j$ do not overlap if $\delta_{ij}$ belongs to the
feasible region outside the corresponding no-fit polygon. Following the
vertical-slice formulation of ~\shortciteA{lastra2024mixed},
this feasible relative-placement region is decomposed into a finite set
of disjoint linear regions. Let $\mathcal{R}_{ij}$ denote the set of
candidate regions for pair $(i,j)$, and let $\eta_{ijr}$ be a binary
variable equal to one if region $r\in\mathcal{R}_{ij}$ is selected. The
oracle imposes
\begin{equation}
    \sum_{r\in\mathcal{R}_{ij}} \eta_{ijr}=1,
    \quad \forall i,j\in S,\ i<j,
    \label{eq:single-bin-region-selection}
\end{equation}
so that exactly one feasible relative-placement region is selected for
each item pair. If $\eta_{ijr}=1$, the corresponding linear inequalities
of region $r$ are enforced on $\delta_{ij}$. These constraints ensure
pairwise non-overlap between all items in $S$.

The resulting model is a feasibility MIP with continuous placement
variables and binary region-selection variables. Its objective is set to
zero:
\begin{equation}
    \min \ 0.
    \label{eq:single-bin-zero-objective}
\end{equation}
If the model is feasible, the oracle returns $\Phi(S)=1$ together with a
valid placement of all items in $S$. If the model is proven infeasible,
the oracle returns $\Phi(S)=0$, and the corresponding item set can be
used to generate a Benders feasibility cut.

Our implementation uses a compact version of the vertical-slice model.
The compact formulation preserves the same feasibility interpretation
but reduces redundant variables and constraints in the pairwise
non-overlap model. In addition, symmetry-breaking constraints are imposed on identical items, while triplet-based feasibility cuts derived from the vertical-slice representation are generated at the root node via a user-cut callback.
 These enhancements are used only to accelerate the exact feasibility check.

Since the single-bin model follows an existing exact MIP formulation for
irregular packing, we do not reproduce all geometric constraints here.
The detailed construction of the vertical-slice decomposition, the
corresponding linear inequalities, and the associated valid cuts can be
found in the work of ~\shortciteA{lastra2024mixed}. 

\subsection{Master Problem Solution}
\label{subsec:master-problem-solution}

The master problem in the proposed Benders framework is a restricted model over item subsets. Given the current cut pool
$\mathcal{F}'$, the RBMP selects a collection of candidate
patterns such that every item is covered exactly once and no selected
pattern contains a known infeasible subset. Therefore, once the geometric
cuts in $\mathcal{F}'$ are fixed, the master problem is essentially a
one-dimensional bin packing problem with additional infeasible-subset
constraints:
\begin{align}
    \min \quad
        & \sum_{p\in\mathcal{P}} \lambda_p
        \label{eq:rbmp-master-obj} \\
    \text{s.t.}\quad
        & \sum_{p\in\mathcal{P}} a_{ip}\lambda_p = 1,
        && \forall i\in I,
        \label{eq:rbmp-master-cover} \\
        & \sum_{p\in\mathcal{P}: S\subseteq S(p)} \lambda_p = 0,
        && \forall S\in\mathcal{F}',
        \label{eq:rbmp-master-cuts} \\
        & \lambda_p\in\{0,1\},
        && \forall p\in\mathcal{P}.
        \label{eq:rbmp-master-binary}
\end{align}
Here, $\mathcal{P}$ denotes the complete set of candidate item subsets satisfying the basic bin-capacity condition. 
Constraint
\eqref{eq:rbmp-master-cover} assigns each item to exactly one selected
pattern, while Constraint~\eqref{eq:rbmp-master-cuts} prevents the master
problem from selecting any pattern that contains a known geometrically
infeasible subset.

Although this model has a simple combinatorial structure, directly
solving it is difficult for the benchmark instances considered in this
paper. Many irregular bin packing benchmarks are puzzle-like instances, where optimal packings can be highly constrained and geometrically feasible item groupings are difficult to identify. As a result, a basic compact
formulation or a direct enumeration of patterns tends to explore a large
number of geometrically invalid or irrelevant item combinations before
reaching the optimal packing structure.

To address this difficulty, we solve the RBMP by a pattern-based
branch-and-price procedure and embed this procedure inside the Benders
decomposition framework. The branch-and-price algorithm avoids
enumerating all patterns in advance. Instead, it starts from a small
initial pattern set and generates promising patterns dynamically through
pricing. Whenever the Benders oracle identifies a new infeasible item
subset, the corresponding cut is added to the RBMP and is also passed to
the pricing problem, so that future generated patterns respect the
current geometric feedback. In this way, the branch-and-price procedure
solves the combinatorial master problem, while the Benders loop
progressively refines the master problem with exact geometric
information.
\subsubsection{Branch-and-Price Framework for RBMP}
\label{subsubsec:branch-and-price-framework}

We solve the RBMP by a pattern-based branch-and-price framework adapted
from the exact branch-and-price-and-cut algorithm of
~\shortciteA{wei2020new} for one-dimensional bin packing and bin
packing with conflicts. In our approach, the role of this procedure is to solve
the restricted Benders master problem under the current geometric cut
pool $\mathcal{F}'$.

At each node of the branch-and-price tree, we solve the linear relaxation 
of the current restricted master problem over a limited column set
$\mathcal{P}'\subseteq\mathcal{P}$. The node relaxation is
\begin{align}
    \min \quad
        & \sum_{p\in\mathcal{P}'} \lambda_p
        \label{eq:node-rmp-obj} \\
    \text{s.t.}\quad
        & \sum_{p\in\mathcal{P}'} a_{ip}\lambda_p = 1,
        && \forall i\in I,
        \label{eq:node-rmp-cover} \\
        & \lambda_p \ge 0,
        && \forall p\in\mathcal{P}'.
        \label{eq:node-rmp-nonnegative}
\end{align}
The above model can be regarded as the restricted master problem in the branch-and-price procedure.
Patterns $\mathcal{P}'$ are added by solving the pricing problems.
Since $\mathcal{P}'$ contains only patterns that are admissible under the current Benders cut pool $\mathcal{F}'$, the corresponding feasibility cuts need not be written explicitly in the node RMP.

After solving the node relaxation, the dual values of the covering
constraints are passed to the pricing problem. The pricing problem
searches for an admissible pattern with negative reduced cost. In its
basic form, this is a one-dimensional knapsack-type problem:
\begin{align}
    \min \quad
        & 1-\sum_{i\in I} \pi_i q_i
        \label{eq:pricing-reduced-cost} \\
    \text{s.t.}\quad
        & \sum_{i\in I} A_i q_i \le A,
        \label{eq:pricing-area} \\
        & \sum_{i\in S} q_i \le |S|-1,
        && \forall S\in\mathcal{F}',
        \label{eq:pricing-benders-cuts} \\
        & q_i\in\{0,1\},
        && \forall i\in I.
        \label{eq:pricing-binary}
\end{align}
Here, $\pi_i$ is the dual value associated with the covering constraint
of item $i$, $A_i$ is the area of item $i$, and $A$ is the bin area.
Constraint~\eqref{eq:pricing-area} ensures that the total area of the items does not exceed the area of the bin.
Constraint~\eqref{eq:pricing-benders-cuts} prevents the pricing problem
from regenerating item combinations that have already been certified
geometrically infeasible. 
These constraints define the admissible pattern space of the current RBMP and thereby embed the geometric feedback from the Benders oracle directly into the column-generation process.
If the pricing problem finds a pattern with
negative reduced cost, the corresponding column is added to
$\mathcal{P}'$, and the node relaxation is solved again. Otherwise, the
current LP relaxation is solved to optimality for the current node.
During the branch-and-price process, branching-induced restrictions on admissible patterns are also incorporated into the pricing problem so that newly generated columns remain consistent with the branching restrictions at the current node.

The pricing problem is solved using a label-setting dynamic programming procedure adapted from the general approach of~\shortciteA{wei2020new}. 
The basic label representation and dominance principles follow standard techniques for this class of pricing problems, while several modifications are introduced to fit the proposed framework. 
In particular, the label-setting procedure incorporates the current Benders cut pool so that labels leading to known infeasible item combinations can be discarded during pricing. 
In addition, an adaptive pricing strategy is employed to balance pricing efficiency and exactness, as described in Section~\ref{subsubsec:adaptive-pricing}.

If the solution of the node relaxation is integer, the selected patterns
define a candidate bin assignment. This candidate is then passed to the
single-bin geometric feasibility oracle described in
Section~\ref{subsec:single-bin-subproblem}. 
If the node relaxation is fractional, we branch on a pair of items. 
 For each item pair $(i,j)$, the algorithm
computes its co-assignment value in the current fractional RMP solution:
\begin{equation}
    \rho_{ij}
    =
    \sum_{p\in\mathcal{P}':\, i,j\in S(p)}
    \lambda_p .
    \label{eq:pair-coassignment-reward}
\end{equation}
A pair with fractional co-assignment value is selected for branching. In
the implementation, the selected pair is the one whose value is closest
to a prescribed target value:
\begin{equation}
    (i^*,j^*)
    \in
    \arg\min_{i\ne j}
    \left|\rho_{ij}-\tau\right|,
    \label{eq:branch-pair-selection}
\end{equation}
where $\tau$ is set to $0.55$ by default. Ties are broken by preferring
larger item-group weight. The two child nodes then enforce, respectively,
that the selected pair must be packed together or must be packed in
different bins. 
The first child enforces that $i$ and $j$ must be packed
together, while the second child enforces that they must be packed in
different bins. 
This branching rule preserves the structure of the pricing problem because the two branching decisions can be represented either by merging item groups or by adding an incompatibility constraint, without changing the basic form of the pricing model.
In addition, several problem-specific enhancements are introduced into the branching procedure to exploit the structure of the RBMP and the geometric feedback mechanism, as detailed in Section~\ref{subsubsec:objective-layered}.

\subsubsection{Objective-Layered Search}
\label{subsubsec:objective-layered}

The branch-and-price procedure is embedded in an objective-layered search
scheme, in which each master-problem solve is restricted to a fixed target
number of bins, allowing branch-and-price nodes whose lower bounds exceed
the current objective layer to be fathomed immediately.
After the initial RBMP is solved, its objective value is rounded
up to obtain an initial lower bound:
\begin{equation}
    LB_0=\left\lceil z_{\mathrm{RBMP}}^0-\varepsilon \right\rceil .
    \label{eq:initial-layer-bound}
\end{equation}
The algorithm then searches the objective values sequentially, starting
from $LB_0$. At each objective layer $k$, the branch-and-price procedure
attempts to find an integer master solution with objective value exactly equal to $k$. If such a solution is found, the selected patterns are
passed to the exact single-bin feasibility oracle. If all selected bins
are geometrically feasible, the solution is accepted as an optimal
solution. Otherwise, the infeasible item sets detected by the oracle are
added as Benders cuts, and the master problem is solved again with the
updated cut pool.
The objective layers are explored in the order
\begin{equation}
    k = LB_0, LB_0+1, LB_0+2,\ldots .
    \label{eq:objective-layer-order}
\end{equation}
In each layer $k$, the optimal objective value of the current restricted master problem is also used for node fathoming. Let
$z_{\mathrm{RMP}}^*$ denote the optimal objective value of the current
RMP, and let $c^*<0$ denote the minimum reduced cost returned by the
exact pricing problem. The current
branch-and-price node can be safely fathomed if
\begin{equation}
    z_{\mathrm{RMP}}^* > k(1-c^*).
\end{equation}
Indeed, this condition is equivalent to
$z_{\mathrm{RMP}}^* + k c^* > k$, while
$z_{\mathrm{RMP}}^* + k c^*$ is a valid lower bound on the full LP
relaxation at the current node. The validity of this fathoming rule is
established in Lemma~1 of Appendix~A.3.
This strategy is useful for the benchmark instances considered in this
paper, where the optimal number of bins is a small integer, but the
corresponding feasible geometric arrangement is highly constrained. By
checking one objective level at a time, the algorithm avoids searching
solutions with larger numbers of bins before all smaller objective levels
have been ruled out.

When an integer RBMP solution is found during the branch-and-price search,
the algorithm records the current search state and passes the selected bins
to the geometric feasibility oracle. If an infeasible bin is detected, the
corresponding infeasible item set is added to the current cut pool, and the
pattern pool is updated accordingly. The strengthened RBMP is then re-solved
from the recorded continuation point, rather than restarting the
branch-and-price procedure from the root. In this way, Benders cut generation
is embedded directly into the branch-and-price search, and the previously
explored search state is retained whenever new geometric information is
introduced.
This integration is valid because adding Benders feasibility cuts cannot
create new feasible master solutions. Let \(\Omega(\mathcal{F}')\) denote the feasible region at an arbitrary node of the branch-and-price tree used to solve the RBMP under the current cut pool \(\mathcal{F}'\).
 If a new infeasible item set $\bar{S}$ is added, then
\begin{equation}
    \Omega(\mathcal{F}'\cup\{\bar{S}\})
    \subseteq
    \Omega(\mathcal{F}').
    \label{eq:cut-pool-monotonicity}
\end{equation}
Therefore, for a fixed objective layer, any part of the branch-and-price tree that has already been fathomed remains fathomed after additional Benders cuts are inserted. 
The cuts only remove candidate
patterns and candidate integer solutions. 
They do not introduce new
patterns or improve a previously excluded branch into a feasible optimal
solution. 
Previously generated columns are likewise reused whenever they remain admissible under the strengthened cut pool, while columns violating newly added cuts are discarded.

\subsubsection{Adaptive-Pricing}
\label{subsubsec:adaptive-pricing}

In addition to branching, the efficiency of the branch-and-price search also depends critically on the pricing procedure.
The pricing problem is solved exactly at every column-generation iteration.
However, the computational performance of the label-setting dynamic programming
algorithm varies substantially across instances and pricing states. We therefore
employ an adaptive exact pricing strategy that switches between two exact
dynamic-programming procedures. The fast procedure is used initially and orders
item groups in non-increasing order of the ratio between their current dual value
and area. The lower-bound-enhanced procedure constructs an additional
lower-bound table and retains the original item-group order \shortcite{wei2020new}. Both procedures
solve the same pricing problem and enforce the same capacity, branching,
Benders-incompatibility, and duplicate-exclusion constraints. Hence, switching
between them affects only the search process and does not alter the feasible
region or the reduced-cost definition.

At the beginning of each instance, the lower-bound-enhanced procedure is
executed twice for timing calibration.  The first execution serves as a warm-up, while the lower-bound-table construction time measured during the second execution is denoted by \(C_{\mathrm{cal}}\).  If this measurement is unavailable or nonpositive, a fallback estimate is computed as
\begin{equation}
    \widehat{C}
    =
    c_{0}
    +
    c_{1}
    \bigl(|G|+1\bigr)
    \bigl(\lfloor WH \rfloor+1\bigr),
    \label{eq:pricing-fallback-cost}
\end{equation}
where \(G\) is the current set of item groups, and \(W\) and \(H\) are the
bin dimensions. Thus, the calibrated unit cost is defined as
\begin{equation}
    U =
    \begin{cases}
        C_{\mathrm{cal}},
        & \text{if } C_{\mathrm{cal}} > 0,\\[2mm]
        \widehat{C},
        & \text{otherwise}.
    \end{cases}
    \label{eq:pricing-unit-cost}
\end{equation}
For the \(t\)-th pricing call, the switching threshold is then defined as
\begin{equation}
    T_t
    =
    \max\left\{
        0,\,
        \alpha t U + \beta
    \right\},
    \label{eq:pricing-switch-threshold}
\end{equation}
with
\[
    c_{0}=0.001,\qquad
    c_{1}=2\times10^{-8},\qquad
    \alpha=0.6,\qquad
    \beta=0.
\]
During the \(t\)-th fast pricing call, let
\(C_{t-1}^{\mathrm{fast}}\) denote the cumulative running time of all
previously completed fast pricing calls and let \(e_t\) denote the elapsed
time of the current call. If
\begin{equation}
    C_{t-1}^{\mathrm{fast}} + e_t > T_t,
    \label{eq:pricing-switch-condition}
\end{equation}
the current fast execution is terminated without using its incomplete result,
and the same pricing problem is restarted from the beginning using the
lower-bound-enhanced procedure, which is then retained for all subsequent
pricing calls of the instance. Since an interrupted fast execution is never
used to certify the absence of a negative reduced-cost column, and the pricing
problem is subsequently solved exactly by the lower-bound-enhanced procedure,
the adaptive switching strategy preserves the exactness of the
column-generation process.

The pricing problem is solved through a binary search structure in which each level corresponds to an item group, and the two child nodes represent whether the next group is selected or excluded. 
To reduce the number of nodes explored, a lightweight optimistic bound is first evaluated for each label before switching to the precomputed knapsack bound. 
Given residual area \(Q\), a fractional-knapsack bound \(\Pi_{\mathrm{frac}}\) is obtained by considering the remaining groups in decreasing dual-value-to-area ratio and allowing the last group to be selected fractionally. Independently, let \(A_{\min}\) be the smallest positive area among the remaining groups. At most \(\lfloor Q/A_{\min}\rfloor\) such groups can be selected; summing the corresponding largest positive dual values gives a cardinality bound \(\Pi_{\mathrm{card}}\). Positive-dual groups with zero area are included in both bounds. We use

$$
\Pi=\min\{\Pi_{\mathrm{frac}},\Pi_{\mathrm{card}}\},
$$
and cache it by the remaining-group index and residual area. 
Both bounds relax the known incompatibility cuts, so \(\Pi\) remains an upper bound on the dual value of any feasible continuation. A label is pruned if the sum of its accumulated dual value and \(\Pi\) does not exceed \(1\), or if the corresponding reduced-cost lower bound cannot improve the best pattern found so far. 
The search procedure based on the precomputed knapsack table follows~\shortciteA{wei2020new}.

\subsection{Exactness }
\label{subsec:exactness-implementation}

The proposed algorithm is exact and returns an optimal solution to the complete Benders formulation in Section~\ref{sec:benders-formulation} upon termination. Since this formulation is equivalent to the original two-dimensional irregular bin packing problem, the returned solution is also globally optimal for the original problem. The main exactness results are stated below, and the detailed proofs are provided in ~\ref{app:exactness-proof}.

\begin{proposition}[Validity of Benders feasibility cuts]
\label{prop:valid-benders-cuts}
Let $S\subseteq I$ be an item subset such that the single-bin feasibility
oracle proves $\Phi(S)=0$. Then the cut
\begin{equation}
    \sum_{p\in\mathcal{P}:S\subseteq S(p)} \lambda_p = 0
    \label{eq:validity-cut}
\end{equation}
is valid for the complete Benders formulation.
\end{proposition}

\begin{proposition}[Previously searched branches remain fathomed]
\label{prop:searched-branches-remain-fathomed}
For a fixed objective layer, adding valid Benders feasibility cuts cannot
make a previously searched and fathomed branch become capable of
producing an optimal solution at that layer.
\end{proposition}

\begin{proposition}[Exactness of complete branch-and-price search]
\label{prop:complete-bp-fixed-cut}
For a fixed cut pool $\mathcal{F}'$, a complete branch-and-price search
solves the corresponding RBMP exactly.
\end{proposition}

\begin{proposition}[Exactness of continued branch-and-price search]
\label{prop:continued-bp-after-cut}
For a fixed objective layer, after adding valid Benders cuts, continuing
the branch-and-price search from the current state over the updated cut
pool preserves exactness.
\end{proposition}

\begin{proposition}[Optimality certificate at an objective layer]
\label{prop:certified-rbmp-optimality}
Suppose that all objective layers smaller than \(k\) have been excluded by exact branch-and-price search, and an integer RBMP solution with objective value \(k\) is certified geometrically feasible by the single-bin oracle. Then this solution is optimal for the complete Benders formulation.
\end{proposition}

\begin{theorem}[Exactness of the proposed algorithm]
\label{thm:algorithm-exactness}
If the proposed algorithm terminates with a solution, then the returned
solution is an optimal solution of the original two-dimensional irregular
bin packing problem.
\end{theorem}

The proof follows from the validity of all generated Benders cuts, the
fact that previously searched branches remain fathomed after cut
insertion, the exactness of complete branch-and-price search for a fixed
cut pool, the exactness of continued branch-and-price search after cut
insertion, and the optimality certificate provided by the single-bin
oracle.

\section{Computational Experiments}
\label{sec:experiments}
\subsection{Experimental Setup}
\label{subsec:experimental-setup}

The experiments are conducted on a benchmark set of two-dimensional
irregular bin packing instances. The benchmark consists of 18 classes,
denoted by TA--TR, and each class contains 30 instances. In total, the
test set contains 540 instances. The instances are obtained from \url{https://github.com/ESICUP/datasets}.
For each instance, the bin dimensions, item polygons, and no-fit-polygon
data are provided as input. 

All experiments are performed on a computer equipped with Core i7-11700@2.50GHz CPU and
64 GB RAM, running Windows 11 operating system. The code is implemented in C++. Gurobi 12 is used for the LP and MIP components
of the algorithm. All runs are executed in single-thread mode. To avoid
interference among instances, the benchmark cases are solved strictly
sequentially, with a one-second waiting interval between two consecutive
runs.
All computational experiments were conducted using Gurobi under a valid academic license that was available during the experimental period.

Each instance is given a time limit of 3600 seconds. The reported running
time is the total wall-clock time of the proposed method, including the preprocessing, 
restricted master problem, pricing, branch-and-price search, Benders cut
generation, and single-bin feasibility checks. An instance is marked as
solved only if the algorithm returns a solution certified optimal by the
Benders framework. If the time limit is reached before such a certificate
is obtained, the instance is reported as unsolved within the time limit.
The detailed experimental configuration is summarized in
Table~\ref{tab:solver-config}.

\begin{table}[H]
\centering
\caption{Algorithmic and solver configuration}
\label{tab:solver-config}
\begin{tabular}{ll}
\hline
Parameter & Value \\
\hline
Runtime & 3600s\\
Number of threads & 1 \\
Branching target $\tau$ & 0.55 \\
Gurobi LP method & Dual simplex \\
Gurobi Crossover & Disabled \\
Gurobi NumericFocus & 2 \\
Gurobi FeasibilityTol & $10^{-8}$ \\
Gurobi OptimalityTol & $10^{-9}$ \\
\hline
\end{tabular}
\end{table}

\subsection{Results on Benchmark Instances}
\label{subsec:benchmark-results}

This section reports the computational results on the benchmark
instances. We first present the performance of the proposed method on all
18 benchmark classes. We then compare the proposed method with two
baseline approaches on the classes where the proposed method solves at
least one instance.

\subsubsection{Overall results of the proposed method.}
Table~\ref{tab:proposed-all-results} summarizes the performance of the
proposed method over all 18 benchmark classes. Column ``\#Opt'' gives
the number of instances solved to obtain the optimal solution within the time
limit. Column ``\#Unsolved'' includes all instances not solved to proven
optimality within the time limit.
Column ``Avg. obj.'' is the average certified objective value over solved
instances. Column ``Avg. time'' is the average wall-clock running time
over solved instances only.

\begin{table}[t]
\centering
\caption{Results of the proposed method on all benchmark classes}
\label{tab:proposed-all-results}
\begin{tabular}{lrrrrr}
\hline
Class & \#Inst & \#Opt & \#Unsolved & Avg. obj. & Avg. time (s)\\
\hline
TA & 30 & 30 & 0  & 3.00  & 222.73 \\
TB & 30 & 30 & 0  & 10.00 & 0.32 \\
TC & 30 & 30 & 0  & 6.00  & 5.75 \\
TD & 30 & 0  & 30 & --    & -- \\
TE & 30 & 0  & 30 & --    & -- \\
TF & 30 & 0  & 30 & --    & -- \\
TG & 30 & 30 & 0  & 13.53 & 2.91 \\
TH & 30 & 30 & 0  & 12.00 & 0.43 \\
TI & 30 & 0  & 30 & --    & -- \\
TJ & 30 & 0  & 30 & --    & -- \\
TK & 30 & 30 & 0  & 6.00  & 169.33 \\
TL & 30 & 30 & 0  & 3.00  & 32.66 \\
TM & 30 & 30 & 0  & 5.00  & 230.96 \\
TN & 30 & 0  & 30 & --    & -- \\
TO & 30 & 30 & 0  & 7.00  & 0.59 \\
TP & 30 & 1  & 29 & 8.00  & 1769.29 \\
TQ & 30 & 30 & 0  & 15.00 & 49.45 \\
TR & 30 & 17 & 13 & 9.00  & 962.34 \\
\hline
Total & 540 & 318 & 222 & -- & -- \\
\hline
\end{tabular}
\end{table}

The proposed method solves 318 out of 540 instances to obtain the optimal solution. It solves all 30 instances in ten classes: TA, TB, TC, TG,
TH, TK, TL, TM, TO, and TQ. Among these classes, TB, TH, and TO are
solved very quickly, with average running times below one second. The
classes TA, TK, TM, and TR are more difficult, and the average running
time over solved instances is larger. No instance is solved
within the time limit for TD, TE, TF, TI, TJ, and TN.

Class TP requires additional care in the implementation. In the reported
TP experiments, weak dominance in the pricing routine is disabled to
avoid memory exhaustion. This is an engineering stability treatment for
the pricing implementation and does not change the mathematical model,
the Benders cuts, or the exactness arguments in
Section~\ref{subsec:exactness-implementation}. 

\subsubsection{Comparison with MIP and CBD.}
The three methods considered in the comparison are as follows. The first
method is the proposed exact combinatorial Benders decomposition
algorithm. The second method, denoted by MIP, is a direct mixed-integer
programming formulation solved by Gurobi. The third method, denoted by
CBD, is a baseline combinatorial Benders decomposition method adapted
from the exact Benders framework originally developed for rectangular bin
packing problems~\cite{cote2021combinatorial}. Since several components of that method are
specific to rectangular items, we remove the rectangle-specific
precomputed lower bounds and lifted cuts, and keep only the transferable
Benders structure.

We next compare the proposed method with MIP and CBD on the 12 classes
where the proposed method solves at least one instance. Table
\ref{tab:comparison-success-classes} reports the comparison. For each
method, ``\#Opt'' is the number of instances solved to obtain the optimal solution.
``Avg. LB'' is the average lower bound, which is computed over all available runs in the class. For all
methods, ``Avg. time'' is the average running time over solved instances
only. For MIP, we additionally report the average upper bound and average
optimality gap over instances where both lower and upper bounds are
available. The gap is computed as
\[
    \mathrm{gap}
    =
    \frac{UB-LB}{|UB|}\times 100\%.
\]
A dash indicates that the corresponding statistic is not available.
To provide a more intuitive comparison of the three methods, Figure~\ref{fig:solved_instances} reports the number of instances solved to obtain the optimal solution in each benchmark class.

\begin{table}[t]
\centering
\scriptsize
\caption{Comparison on classes where the proposed method solves at least one instance}
\label{tab:comparison-success-classes}
\begin{tabular}{lrrr|rrr|rrrrr}
\hline
Class
& \multicolumn{3}{c|}{Proposed}
& \multicolumn{3}{c|}{CBD}
& \multicolumn{5}{c}{MIP} \\
\cline{2-4}\cline{5-7}\cline{8-12}
& \#Opt & Avg. LB & Avg. time (s)
& \#Opt & Avg. LB & Avg. time (s)
& \#Opt & Avg. LB & Avg. UB & Avg. gap & Avg. time (s) \\
\hline
TA & 30 & 3.00  & 222.73 & 3  & 3.00  & 481.50  & 0  & 3.00  & 4.17  & 27.50 & -- \\
TB & 30 & 10.00 & 0.32   & 30 & 10.00    & 0.64    & 30 & 10.00 & 10.00 & 0.00  & 8.69 \\
TC & 30 & 6.00  & 5.75   & 2  & 6.00  & 1358.52 & 0  & 6.00  & 8.10  & 25.67 & -- \\
TG & 30 & 13.53 & 2.91   & 17 & 13.20   & 549.68  & 25 & 13.40 & 13.57 & 1.21  & 95.45 \\
TH & 30 & 12.00 & 0.43   & 30 & 12.00 & 2.08    & 30 & 12.00 & 12.00 & 0.00  & 30.29 \\
TK & 30 & 6.00  & 169.33 & 0  & 6.00  & --      & 0  & 5.97  & 8.31  & 27.97 & -- \\
TL & 30 & 3.00  & 32.66  & 28 & 3.00  & 684.03  & 3  & 3.00  & 4.10  & 25.50 & 416.49 \\
TM & 30 & 5.00  & 230.96 & 0  & 5.00  & --      & 0  & 4.93  & 7.23  & 30.98 & -- \\
TO & 30 & 7.00  & 0.59   & 30 & 7.00  & 0.32    & 30 & 7.00  & 7.00  & 0.00  & 9.55 \\
TP & 1  & 8.00  & 1769.29& 0  & 8.00  & --      & 0  & 7.60  & 13.93 & 44.19 & -- \\
TQ & 30 & 15.00 & 49.45  & 0  & 15.00 & --      & 0  & 14.13 & 25.93 & 45.33 & -- \\
TR & 17 & 9.00  & 962.34 & 0  & 9.00    & --      & 0  & 7.33  & 17.00 & 56.72 & -- \\
\hline
\end{tabular}
\end{table}
\begin{figure}[htbp]
    \centering
    \includegraphics[width=\linewidth]{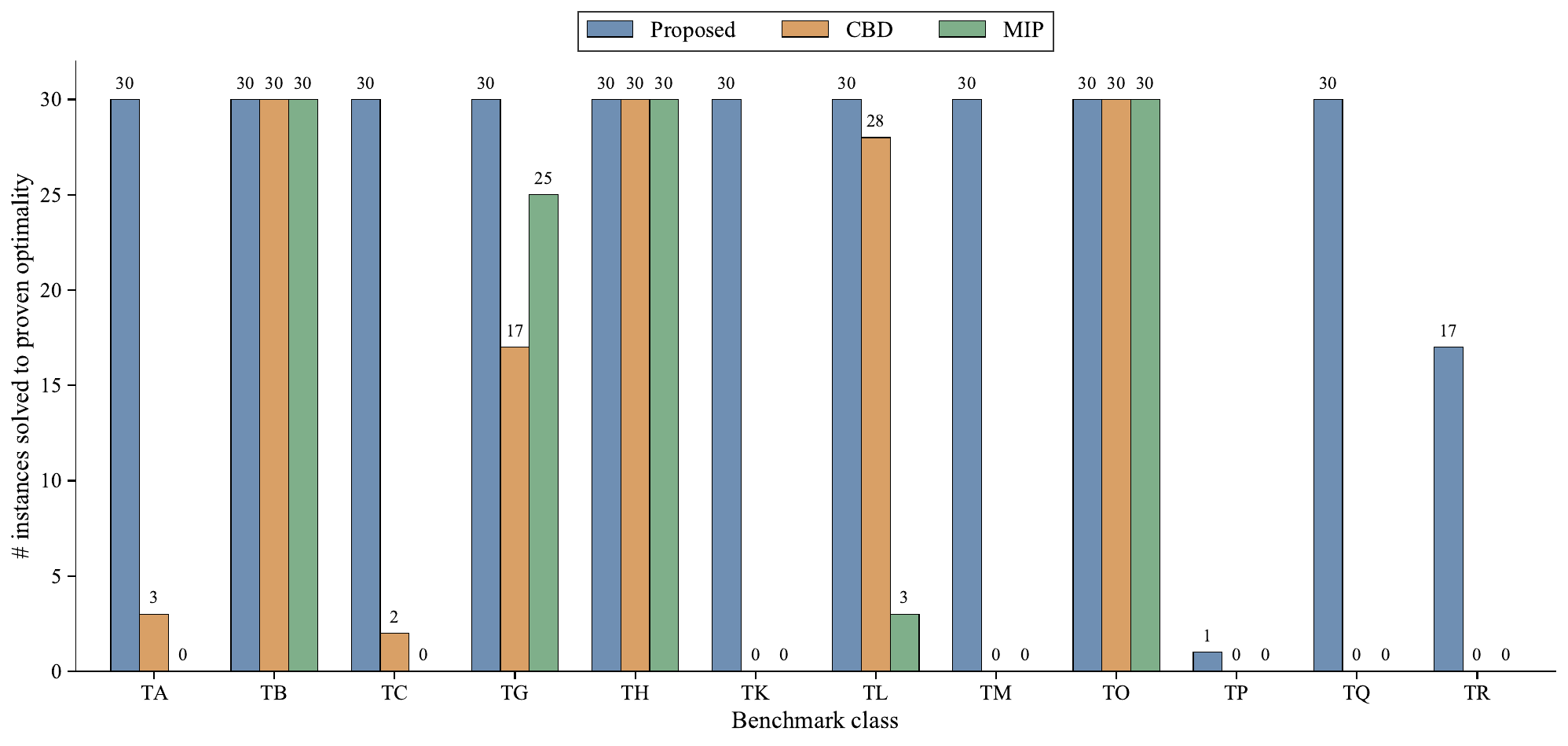}
    \caption{Number of instances solved to obtain the optimal solution within 3600 s by the proposed method and the baseline approaches across benchmark classes.}
    \label{fig:solved_instances}
\end{figure}

The comparison shows that the proposed method outperforms the two
baselines on most classes where it is effective. CBD performs well on
TB, TH, and TO, but its performance deteriorates on classes such as TA,
TC, TK, TM, TP, TQ, and TR. The proposed method solves more instances than the baseline CBD on TA, TC, TK, TM, TQ, and TR.
In particular, solving the master problem is a major computational challenge within the Benders framework, and the results highlight the benefit of introducing a tailored branch-and-price procedure to improve the efficiency of the master search.
The direct MIP approach also solves the relatively easy classes TB, TH,
and TO, and solves many TG instances. However, it struggles to close the
optimality gap on the more geometrically constrained classes. This is
reflected by the large MIP gaps on TA, TC, TK, TL, TM, TP, TQ, and TR. On
these classes, Gurobi often finds feasible upper bounds but does not
prove optimality within the time limit.
This comparison further highlights the value of the Benders decomposition strategy, which separates the original problem into a combinatorial master problem and an exact single-bin geometric feasibility subproblem, rather than requiring the solver to handle the full combinatorial and geometric complexity simultaneously.

For TP, the direct MIP solver produced incumbent solutions in some runs,
but the reported lower and upper bounds did not match at termination.
These runs are therefore not counted as solved. We report their lower
bounds, upper bounds, and gaps only as bound information. More generally,
a MIP run is counted as solved only when a valid optimality certificate
is obtained. Runs that return an incumbent solution without a matching
lower bound are not counted as solved. This status convention avoids
over-reporting the performance of the direct MIP approach on instances
where a feasible packing is found but optimality is not certified.

The unsolved instances also reveal two main limitations of the current framework. First, the exact single-bin geometric feasibility check remains a major computational bottleneck. In our experiments, oracle calls become particularly expensive when a candidate bin contains more than about 15 polygons. 
While geometric feasibility can sometimes be established quickly by finding a valid placement, infeasibility must be certified by the exact oracle and therefore cannot be safely concluded from a heuristic procedure. 
As a result, difficult infeasible subsets may require substantial computational effort before a valid Benders cut can be generated. 
Second, as the number of items increases, the master problems  become increasingly challenging. 
The RBMP and its pricing problems effectively form  large puzzle-like one-dimensional bin-packing problems with dynamically accumulated conflict and infeasible-subset constraints, leading to a rapidly expanding combinatorial search space. 
These two sources of difficulty—exact geometric infeasibility certification and large-scale combinatorial master search—constitute the main bottlenecks of the current exact framework.
\section{Conclusions}
\label{sec:conclusion}
This paper developed an exact combinatorial Benders decomposition framework for the two-dimensional irregular bin packing problem with convex polygons, coupling a pattern-based branch-and-price master problem with an exact single-bin geometric feasibility oracle. The framework progressively incorporates geometric infeasibility through valid Benders cuts, while objective-layered search, adaptive exact pricing, and geometric feedback improve the efficiency of the dynamically strengthened master search without compromising exactness. Computational experiments on 540 benchmark instances from 18 classes show that the proposed method obtains the optimal solution for 318 instances within a 3600-second time limit and outperforms a direct mixed-integer programming formulation and a baseline combinatorial Benders approach on several challenging classes. 
The computational results demonstrate the effectiveness of separating the combinatorial master search from geometric feasibility certification and coordinating the two components through Benders feedback.
Nevertheless, a substantial number of instances remain unsolved within the time limit, indicating that further improvements in the master search and geometric feasibility evaluation are still needed for the most challenging instances.

Several extensions also deserve further investigation. 
The current study considers convex polygons with fixed orientations. Extending the framework to more general irregular packing settings, including non-convex polygons and multiple allowable rotations, would broaden its applicability but would require more sophisticated geometric representations and feasibility-oracle techniques. 
In addition, several decisions in the current framework, such as Benders-cut selection and branching-candidate selection, are governed by predefined rules.
Learning-based methods could potentially be used to guide these decisions based on the evolving search state. 
An important research direction is therefore to investigate how such data-driven guidance can be integrated into the exact framework while preserving the validity of bounds and cuts and maintaining optimality guarantees.
\bibliography{reference}

\appendix
\section{Proofs of Exactness}
\label{app:exactness-proof}

\subsection{Proof of Proposition~\ref{prop:valid-benders-cuts}}

Let $S\subseteq I$ be an item subset such that the single-bin feasibility
oracle proves $\Phi(S)=0$. By definition of $\Phi$, the items in $S$
cannot be packed together into one bin. Therefore, no feasible solution
of the original problem can contain a selected bin whose item set includes
$S$.

In the pattern-based formulation, a pattern $p$ represents one bin, and
$S(p)$ denotes the item set contained in this pattern. If
$S\subseteq S(p)$, then selecting pattern $p$ would assign all items in
$S$ to the same bin, which is impossible because $\Phi(S)=0$. Hence every
pattern containing $S$ must be forbidden. This is exactly enforced by
\begin{equation}
    \sum_{p\in\mathcal{P}:S\subseteq S(p)} \lambda_p = 0.
\end{equation}
Thus, the cut is valid for the complete Benders formulation. The argument
does not require $S$ to be a minimal infeasible subset; minimality affects
only the strength of the cut, not its validity.
\qed

\subsection{Proof of Proposition~\ref{prop:searched-branches-remain-fathomed}}

Consider a fixed objective layer $k$ and a branch-and-price search tree
defined under the current Benders cut pool $\mathcal{F}'$. Let
$\bar{S}\in\mathcal{F}$ be a newly detected infeasible item set. Adding
the corresponding Benders cut removes all patterns containing
$\bar{S}$. Therefore, if $\Omega(\mathcal{F}')$ denotes the feasible
region of the RBMP under cut pool $\mathcal{F}'$, then
\begin{equation}
    \Omega(\mathcal{F}'\cup\{\bar{S}\})
    \subseteq
    \Omega(\mathcal{F}').
\label{eq:proof-cut-monotonicity}
\end{equation}
The feasible region can only shrink after cut insertion.

Now consider a branch that has already been searched and fathomed before
the new cut is inserted. Such a branch is fathomed only if one of the
following conditions holds: the branch is infeasible, its lower bound is
larger than the current objective layer or incumbent threshold, or all
integer solutions in the branch have already been explored and rejected.
After adding the new Benders cut, the set of feasible columns and
therefore the set of feasible master solutions in this branch is a subset
of the previous one. Hence an infeasible branch cannot become feasible, a
branch whose lower bound was too large cannot acquire a better feasible
solution, and an already exhausted branch cannot contain a new integer
solution that was absent before.

Therefore, adding valid Benders feasibility cuts cannot make a previously
searched and fathomed branch become capable of producing an optimal
solution at the same objective layer.
\qed

\subsection{Proof of Proposition~\ref{prop:complete-bp-fixed-cut}}

Fix a Benders cut pool $\mathcal{F}'$. Under this fixed cut pool, the RBMP
is a set-partitioning master problem over all admissible patterns, where
a pattern is admissible if it satisfies the capacity constraint and does
not contain any item subset in $\mathcal{F}'$.

We first establish the validity of the column-generation bound used at a
branch-and-price node.

\begin{lemma}[Column-generation lower bound]
\label{lem:cg-lower-bound}
Let $k$ be the current objective layer and let $\Omega_f$ denote the set of all columns satisfying the capacity constraint, all current incompatibility constraints, and all branching restrictions imposed at the current branch-and-price node.
Let \(a_j\) denote the item-incidence vector of column \(j\), and let \(b\) denote the right-hand-side vector of the covering constraints.
 Consider the full LP master problem
\begin{equation}
z_{MP,f}^*
=
\min \left\{
\sum_{j\in\Omega_f} x_j :
\sum_{j\in\Omega_f} a_j x_j=b,\;
0\le x_j\le1,\ \forall j\in\Omega_f
\right\}.
\label{eq:MPf-proof}
\end{equation}
Let $R\subseteq\Omega_f$ be the current column set in the restricted
master problem:
\begin{equation}
z_{RMP}^*
=
\min \left\{
\sum_{j\in R}x_j:
\sum_{j\in R}a_jx_j=b,\;
0\le x_j\le1,\ \forall j\in R
\right\}.
\label{eq:RMP-proof}
\end{equation}
Let $\pi^*$ be an optimal dual solution associated with the covering
constraints of the RMP, and let $\beta_j^*\ge0$ be the optimal dual
multipliers associated with the upper bounds $x_j\le1$ for $j\in R$.
Define
\begin{equation}
c^*
=
\min_{j\in\Omega_f\setminus R}
(1-a_j^\top\pi^*).
\label{eq:cstar-def-proof}
\end{equation}
Assume that the pricing routine returns this minimum whenever $c^*<0$,
after excluding duplicate columns and infeasible columns through
additional exclusion constraints. If there exists an optimal solution
$x^{MP,f,*}$ of \eqref{eq:MPf-proof} satisfying
\begin{equation}
    \sum_{j\in\Omega_f}x_j^{MP,f,*}\le k,
\end{equation}
then, whenever $c^*<0$,
\begin{equation}
    z_{RMP}^*+kc^*
    \le
    z_{MP,f}^* .
\label{eq:cg-bound-proof}
\end{equation}

\end{lemma}

\begin{proof}
Since $(\pi^*,\beta^*)$ is dual optimal for
\eqref{eq:RMP-proof}, strong duality gives
\begin{equation}
z_{RMP}^*
=
(\pi^*)^\top b
-
\sum_{j\in R}\beta_j^* .
\label{eq:dual-rmp-proof}
\end{equation}
For each generated column $j\in R$, dual feasibility gives
\begin{equation}
    a_j^\top\pi^*-\beta_j^* \le 1.
\label{eq:dual-generated-proof}
\end{equation}
For each non-generated feasible column
$j\in\Omega_f\setminus R$, the definition of $c^*$ gives
\begin{equation}
    1-a_j^\top\pi^* \ge c^*,
\end{equation}
and therefore
\begin{equation}
    a_j^\top\pi^*+c^* \le 1.
\label{eq:dual-nongenerated-proof}
\end{equation}

Consider any feasible solution $x$ of the full master problem. For
columns in $R$, multiplying \eqref{eq:dual-generated-proof} by $x_j\ge0$
gives
\begin{equation}
    x_j \ge a_j^\top\pi^* x_j-\beta_j^*x_j.
\end{equation}
Since $0\le x_j\le1$ and $\beta_j^*\ge0$, we have
$-\beta_j^*x_j\ge -\beta_j^*$. Hence
\begin{equation}
    x_j \ge a_j^\top\pi^* x_j-\beta_j^*,
    \quad \forall j\in R.
\end{equation}
For columns in $\Omega_f\setminus R$, multiplying
\eqref{eq:dual-nongenerated-proof} by $x_j\ge0$ gives
\begin{equation}
    x_j \ge a_j^\top\pi^* x_j+c^*x_j.
\end{equation}

Summing the above inequalities over all columns yields
\begin{align}
\sum_{j\in\Omega_f}x_j
&\ge
\sum_{j\in R}
(a_j^\top\pi^*x_j-\beta_j^*)
+
\sum_{j\in\Omega_f\setminus R}
(a_j^\top\pi^*x_j+c^*x_j)
\nonumber\\
&=
(\pi^*)^\top
\sum_{j\in\Omega_f}a_jx_j
-
\sum_{j\in R}\beta_j^*
+
c^*
\sum_{j\in\Omega_f\setminus R}x_j
\nonumber\\
&=
(\pi^*)^\top b
-
\sum_{j\in R}\beta_j^*
+
c^*
\sum_{j\in\Omega_f\setminus R}x_j
\nonumber\\
&=
z_{RMP}^*
+
c^*
\sum_{j\in\Omega_f\setminus R}x_j .
\label{eq:bound-general-proof}
\end{align}

Let $x=x^{MP,f,*}$ be an optimal solution of the full master problem.
Since
\[
    z_{MP,f}^*
    =
    \sum_{j\in\Omega_f}x_j^{MP,f,*},
\]
substituting $x^{MP,f,*}$ into \eqref{eq:bound-general-proof} gives
\begin{equation}
z_{MP,f}^*
\ge
z_{RMP}^*
+
c^*
\sum_{j\in\Omega_f\setminus R}
x_j^{MP,f,*}.
\end{equation}
Because $c^*<0$ and
\[
\sum_{j\in\Omega_f\setminus R}
x_j^{MP,f,*}
\le
\sum_{j\in\Omega_f}
x_j^{MP,f,*}
\le k,
\]
we have
\[
c^*
\sum_{j\in\Omega_f\setminus R}
x_j^{MP,f,*}
\ge
kc^* .
\]
Therefore,
\[
z_{MP,f}^*
\ge
z_{RMP}^*+kc^*,
\]
which proves the main bound.

Finally, because the objective function of \eqref{eq:MPf-proof} is
$\sum_{j\in\Omega_f}x_j$, every optimal solution satisfies
\[
\sum_{j\in\Omega_f}x_j^{MP,f,*}=z_{MP,f}^*.
\]
In the objective-layered search, $k$ is the current target objective value.
If the current node can contain an integer solution at objective layer $k$, then the corresponding full LP relaxation necessarily satisfies
\[
z^{*}_{\mathrm{MP},f} \le k.
\]
Hence, condition~(A.6) holds with the known value $k$.
\end{proof}

We now prove the proposition. At each node of the branch-and-price tree,
column generation solves the LP relaxation of the node master problem.
If the exact pricing problem returns no negative reduced-cost column,
then $c^*\ge0$. In this case, all non-generated admissible columns have
nonnegative reduced cost with respect to the current dual solution.
Therefore, the current RMP solution is optimal for the full LP relaxation
over $\Omega_f$.

If $c^*<0$, the pricing problem returns a most negative reduced-cost
column among all non-generated admissible columns. This column is added
to the RMP and column generation continues. Lemma~\ref{lem:cg-lower-bound}
also provides a valid lower bound on the full LP value at the node. Thus,
node fathoming based on this bound cannot discard a branch that contains
an improving integer solution.

If the node LP solution is fractional and cannot be fathomed, the
algorithm branches on a pair of items. The two child nodes enforce,
respectively, that the selected pair must be packed together or must be
packed separately. These two cases cover all integer solutions of the
parent node. The together branch is represented by merging item groups,
and the separate branch is represented by adding an incompatibility
constraint. Hence the branching rule preserves the structure of the
pricing problem and does not remove any feasible integer solution.

Therefore, exact column generation at each node, valid node fathoming,
and exhaustive pair-based branching imply that a complete
branch-and-price search solves the RBMP exactly for the fixed cut pool
$\mathcal{F}'$.
\qed

\subsection{Proof of Proposition~\ref{prop:continued-bp-after-cut}}

Consider a fixed objective layer and suppose that a valid Benders cut is
inserted during the branch-and-price search. Let $\mathcal{T}$ be the
complete branch-and-price tree that would be searched from scratch under
the updated cut pool. This tree can be partitioned into two parts:
branches that have already been searched before the cut insertion, and
branches that have not yet been searched.

By Proposition~\ref{prop:searched-branches-remain-fathomed}, none of the
already searched and fathomed branches can become capable of producing an
optimal solution at the same objective layer after the new Benders cut is
added. Therefore, excluding these already searched branches from the
continued search cannot remove an optimal solution of the updated RBMP.

It remains to consider the unsearched branches. After the cut insertion,
the master problem and the pricing problem are solved with the enlarged
cut pool. Thus, all columns generated in the remaining search satisfy the
updated incompatibility constraints, and the pricing problem is exact over
the remaining admissible columns. By
Proposition~\ref{prop:complete-bp-fixed-cut}, a complete search over the
updated cut pool would solve the updated RBMP exactly. Since the only
part omitted by the continued search consists of branches that cannot
contain an optimal solution after cut insertion, the continued search is
equivalent, with respect to optimality, to the complete search under the
updated cut pool.

Therefore, after adding valid Benders cuts, continuing the
branch-and-price search from the current state over the updated cut pool
preserves exactness.
\qed

\subsection{Proof of Proposition~\ref{prop:certified-rbmp-optimality}}

Let $z^*$ denote the optimal objective value of the complete Benders
formulation. Suppose that an integer RBMP solution with objective value
$k$ is certified geometrically feasible by the exact single-bin oracle.
Since every selected pattern can be packed into one bin, this solution is
feasible for the complete Benders formulation. Therefore,
\begin{equation}
    z^* \le k.
    \label{eq:prop5-upper}
\end{equation}

It remains to show that no feasible solution using fewer than $k$ bins
exists. Suppose, by contradiction, that $z^*=\ell<k$. Then there exists
a feasible solution of the complete Benders formulation using exactly
$\ell$ bins. Because all Benders cuts generated by the algorithm are
valid, no such cut can remove this feasible solution.
Hence, the corresponding integer master solution remains feasible for every RBMP encountered while objective layer \(\ell\) is being searched, including the final RBMP whose exact branch-and-price search would be used to exclude this objective layer.

The branch-and-price procedure is exact at each fixed cut pool, and the
continued search remains exact after valid Benders cuts are inserted.
Therefore, objective layer $\ell$ cannot be completely excluded while
such a feasible solution exists. This contradicts the assumption that
all objective layers smaller than $k$ have already been excluded.

Consequently,
\begin{equation}
    z^* \ge k.
    \label{eq:prop5-lower}
\end{equation}
Combining \eqref{eq:prop5-upper} and \eqref{eq:prop5-lower} yields
\[
    z^* = k.
\]
Thus, the geometrically certified solution at objective layer $k$ is
optimal for the complete Benders formulation.

\subsection{Proof of Theorem~\ref{thm:algorithm-exactness}}

The algorithm starts from a valid restricted Benders master problem and
generates only valid Benders feasibility cuts by
Proposition~\ref{prop:valid-benders-cuts}. Whenever cuts are added during
the branch-and-price search, already searched branches remain safely
fathomed by Proposition~\ref{prop:searched-branches-remain-fathomed}, and
continuing the search over the updated cut pool preserves exactness by
Proposition~\ref{prop:continued-bp-after-cut}. For any fixed cut pool,
the underlying complete branch-and-price search is exact by
Proposition~\ref{prop:complete-bp-fixed-cut}.

The objective-layered search examines bin counts in nondecreasing order,
starting from a valid lower bound. By Propositions 1--4, each smaller
objective layer is excluded exactly despite the dynamic insertion of
Benders cuts. When the algorithm finds an integer solution at objective
layer $k$ whose selected patterns are all certified geometrically feasible,
Proposition~\ref{prop:certified-rbmp-optimality} implies that this solution is optimal for the complete
Benders formulation.

Finally, the complete Benders formulation is equivalent to the original
two-dimensional irregular bin packing problem. Every feasible packing
induces a feasible set of selected patterns, and every selected pattern
certified feasible by the oracle corresponds to a valid single-bin
geometric packing. Hence the returned solution is an optimal solution of
the original problem.
\qed

  \end{document}